\documentclass[10pt]{amsart}
\usepackage{amssymb, amsmath}
\usepackage{graphics}
\usepackage{latexsym}
\usepackage{amsmath}
\usepackage{amssymb,amsthm,amsfonts}
\usepackage{amscd}
\usepackage[arrow, matrix, curve]{xy}
\usepackage{syntonly}
\usepackage{yfonts}
\usepackage{tikz-cd}
\usepackage{filecontents}
\usepackage{mathtools}
\usepackage{stmaryrd}
\usepackage{MnSymbol}
\DeclareMathOperator{\spec}{Spec}

\DeclareMathOperator{\aut}{Aut}
\DeclareMathOperator{\ord}{ord}
\DeclareMathOperator{\gl}{GL}

\DeclareMathOperator{\Ii}{im}

\DeclareMathOperator{\br}{Br}
\DeclareMathOperator{\ram}{Ram}
\DeclareMathOperator{\deck}{Deck}
\DeclareMathOperator{\mi}{\textbf{m}}
\theoremstyle{plain}
\newtheorem{thm}{Theorem}[section]
\newtheorem{theorem}[thm]{Theorem}
\newtheorem{lemma}[thm]{Lemma}
\newtheorem{corollary}[thm]{Corollary}
\newtheorem{proposition}[thm]{Proposition}

\theoremstyle{definition}

\newtheorem{remark}[thm]{Remark}

\newtheorem{definition}[thm]{Definition}

\numberwithin{equation}{thm}

\newcommand{\sC}{{\mathcal C}}

\newcommand{\sF}{{\mathcal F}}

\newcommand{\sL}{{\mathcal L}}

\newcommand{\sO}{{\mathcal O}}

\newcommand{\sT}{{\mathcal T}}
\newcommand{\sU}{{\mathcal U}}
\newcommand{\sV}{{\mathcal V}}

\newcommand{\sY}{{\mathcal Y}}

\newcommand{\C}{{\mathbb C}}

\renewcommand{\L}{{\mathbb L}}

\renewcommand{\P}{{\mathbb P}}
\newcommand{\Q}{{\mathbb Q}}

\newcommand{\Z}{{\mathbb Z}}

\newcommand{\End}{{\rm End}}
\newcommand{\Hom}{{\rm Hom}}

\newcommand{\SL}{{\rm SL}}

\begin{document}
\title[Galois covers of curves and their Families]{On Galois coverings of curves and their Families}
\author{Abolfazl Mohajer}
\address{School of Mathematical and Statistical Sciences, University of Galway, Galway, Ireland.}
\email{abmohajer83@gmail.com}
\subjclass{14A10, 14A15, 14E20, 14E22}
\keywords{Algebraic variety, Galois covering, ramified covering}
\maketitle
\begin{abstract}
In this paper, we describe Galois covers of algebraic curves and their families by using local systems associated to push-forward of sheaves by the structure morphism. More precisely, if $f:C\to Y$, we consider the sheaves $f_*(\C)$. The group action by the Galois group $G$, yields a decomposition of this sheaf into irreducible local systems corresponding to  irreducible representations of the group $G$. If $\rho$ is such an  irreducible representation, the eigensheaf $\L_{\rho}$ of $f_*(\C)$ gives rise to another useful sheaf which is related to the homology group $H_1(C,\C)$. Using this, we describe the action of the Galois group $G$ on the homology group. As a particular example, we study the Dihedral covers of $\P^1$ in some detail. 
\end{abstract}

\section{Introduction}
In this paper, we investigate the structure of Galois covers $f\colon C\to Y$ of smooth complex algebraic curves. The structure of dihedral covers of algebraic varieties has been studied by Catanese and Perroni in \cite{CP} in some detail. The irreducibility of the space of dihedral covers of $\P^1$ of a given numerical type had already been explored by Catanese, Lönne and Perroni in \cite{CLP}. The same authors describe in \cite{CLP2} the irreducible components of the moduli space of dihedral covers of algebraic curves. These papers build the first milestones toward the full description of general non-abelian covers of algebraic varieties. In particular, in \cite{CP}, the algebraic “building data” on an algebraic variety $Y$ which are equivalent to the existence of such covers have been determined. Inspired by these results, the author described in \cite{M} the structure of metabelian Galois coverings of algebraic varieties, generalizing the results of \cite{CP} and in particular their building data. For some applications of such results in another direction, see for example the paper \cite{P21} by Perroni. In the present paper, we take another route and prove some results by constructing and studying local systems on algebraic curves using the structure of Galois covers. We will also look at dihedral coverings of the projective line $\P^1$ using the previously proven results on such coverings. In addition to aforementioned papers, we will use several other results for instance from \cite{Mir}, \cite{R} or \cite{CW} in our investigations. Our results also deal with the monodromy of Galois covers and the line bundles associated to them. As stated in \cite{CP}, a main issue in the classification theory of algebraic varieties is the construction of interesting and illuminating examples. For instance in the book of Enriques \cite{En}, a recurrent method is considering the minimal resolution $S$ of double covers of $\P^2$. Also, as stated in \cite{CP}, we can describe an algebraic function on an algebraic curve $Y$ as a rational function on a projective curve $X$ admitting a non-constant morphism $X\to Y$ such that the field extension $\C(Y)\subset\C(X)$ is generated by our algebraic function. The algebraic function is in general polydromic, i.e., many valued as a function on $Y$, and going around a closed loop we do not return to the same value. It is a theorem of Weierstrass that our algebraic function is a rational function on $Y$ iff it is monodromic, i.e., there is no polydromy (as is mentioned  in \cite{CLP}, here what should be called polydromy is nowadays called monodromy based on the explanation Lucus a non lucendo!). Note that for the type of questions that we are investigating, the Riemann’s existence theorem, Theorem ~\ref{RET} plays a crucial role.

\section{Galois coverings of curves}
Let $C,Y$ be complex smooth projective algebraic curves (equivalently Riemann surfaces) and let $f\colon C\to Y$ be a Galois covering of degree $n$. By this we mean precisely that there exists a finite group $G$ with $|G|=n$ together with a faithful action of $G$ on $C$ such that $f$ realizes $Y$ as the quotient of $C$ by $G$. We denote by $\br(f)\subset Y$ the set of branch points of $f$ and with $\ram(f)\subset C$ the set of ramification points. Note that $\ram(f)$ consists precisely of the points in $C$ with non-trivial stabilizers under the action of $G$ and by definition, $\br(f)=f(\ram(f))$. Since the cover is assumed to be Galois, we also have that $f^{-1}(\br(f))=\ram(f)$. Indeed if $x\in C$ is a ramification point with ramification index $e$, then so are all points in the fiber $f^{-1}(f(x))$. So it makes sense to attribute the ramification index also to the branch points of $f$. In fact we will do this throughout the whole manuscript.  The deck transformation group $\deck(C/ Y)$, i.e., the group of those automorphisms of $C$ that are compatible with $f$ is isomorphic to the Galois group $G$ and acts transitively on each fiber $f^{-1}(y)$. Moreover, the stabilizers of the points of $C$ are conjugate cyclic subgroups of $\deck(C/ Y)$ (or by using the isomorphism $\deck(C/Y)\cong G$, of $G$) of  order $e$, see \cite{Sz}, Proposition 3.2.10. If $x\in f^{-1}(y)$, then any other point $x^{\prime}\in f^{-1}(y)$ is of the form $x^{\prime}=\sigma\cdotp x$ for some $\sigma\in \deck(C/ Y)\cong G$. If $H_x=\langle h_x\rangle$ denotes the stabilizer of $x$, then $H_{x^{\prime}}=\langle h_{x^{\prime}}\rangle=\sigma H_{x}\sigma^{-1}$. Now assume that $y$ is a branch point of $f$ of ramification index $e$ and $x\in f^{-1}(y)$. By the above, the stabilizer $H_x=\langle h_x\rangle$ is then a non-trivial cyclic group of order $e$. Fix a primitive $e$-th root of unity $\xi_e$. Consider the primitive character $\eta_x:H_x\to \C^*$ and suppose that $h_x$ is the unique element such that $\eta_x(h_x)=\xi_e$. We call the element $h_x$ \emph{the local monodromy} at the ramification point $x$. If $x^{\prime}\in f^{-1}(y)$  is any other point, then the local monodromy at $x^{\prime}$ is $h_{x^{\prime}}=\sigma\cdotp h_{x}\cdotp\sigma^{-1}$.\par In particular the stabilizer of a point in $C$ is trivial, if and only if that point is \emph{not} a ramification point. The stabilizer $H_x$ of a point $x\in C$ is called the \emph{inertia subgroup} of $x$. Let $\pi_1(Y\setminus\br(f))$ be the fundamental group of the punctures Riemann surface. By the theory of Riemann surfaces, one obtains a so-called \emph{monodromy (permutation) representation} $\Phi:\pi_1(Y\setminus\br(f))\to S_n$. In the sequel we usually denote $\Delta=\br(f)$ and $|\Delta|=r$. The celeberated result of Riemann, the \emph{Riemann existence theorem}, asserts that the covering $f$ is fully determined by the homomorphism $\Phi$. If $y$ is a branch point of $f$ with ramification index $e_y$, then the cycle structure of the permutation $\gamma$ representing a small loop around the branch point $y$ is $(\underbrace{e_y,\dots,e_y}_{\frac{n}{e_y}-\text{times}})$. Indeed, at every ramification point over $y$, there is a cycle of length $e_y$ and the number of ramification points is $\frac{n}{e_y}=|f^{-1}(y)|$. Recall that 
\begin{equation} \label{braid1}
\pi_1(Y\setminus\Delta)\coloneqq \Gamma_{g^{\prime},r}=\langle \alpha_{1}, \beta_1, \ldots, \alpha_{g^{\prime}}, \beta_{g^{\prime}}, \gamma_1,\ldots, \gamma_r\mid \prod_1^r\gamma_i\prod_1^{g^{\prime}}[\alpha_{j}, \beta_j]=1 \rangle,
\end{equation}
where $g^{\prime}$ is the genus of $Y$ and $r=|\Delta|$ as above. If $C\to Y$ is a Galois cover of curves with branch locus $\Delta$, let us set $U:=Y\setminus\Delta$ and $V:=f^{-1}(U)$. Then $f|_V:V\to U$ is an unramified Galois covering. If $G$ is the Galois group of $f$, then there is an epimorphism $\pi_1(U)\to G$ or in other words an epimorphism $\Gamma_{g^{\prime},r}\to G$. We therefore make following defintion.
\begin{definition}\label{datum}
A datum is a triple $(\mi, G, \Phi)$, where $\mi=(m_1,\ldots, m_r)$ is an $r$-tuple of integers $m_i\geq 2$, $G$ is a finite group and $\Phi:\Gamma_{g^{\prime},r}\to G$ is an epimorphism such that $\Phi(\gamma_i)$ has order $m_i$ for each $i$.
\end{definition}
The Riemann’s existence theorem asserts that a cover $C\to Y$ can be completely determined by a datum.

\begin{theorem}\label{RET}
A datum compeltely determines a branched cover of curves $C\to Y$. In other words, a finite set $\Delta\subset Y$ and datum as in Definition ~\ref{datum} determines a cover $C\to Y$ up to isomorphism. 
\end{theorem}
In the case $Y=\P^1$, we consider a finite subset $\Delta=\{t_1,\dots, t_r\}\subset \mathbb{P}^1$. It is then well-known that the fundamental group $\pi_1(\P^1\setminus\Delta)$ has the following presentation
\begin{equation} \label{braid}
\pi_1(\P^1\setminus\Delta)\coloneqq \Gamma_r=\langle g_1,\dots, g_r\mid g_1\cdots g_r=1 \rangle
\end{equation}
Recall from the above that if $y\in Y$ is a branch point and $x,x^{\prime}\in f^{-1}(y)$, then the inertia groups $H_x$ and $H_{x^{\prime}}$ are conjugate cyclic subgroups of $G$ and the local monodromies are also conjugate by the same element of $G$. Therefore every finite Galois covering gives rise to a family of conjugacy classes of the finite group $G$. Combining this with \ref{braid} yields the following special case of the Riemann existence theorem.
\begin{theorem}(Riemann existence theorem for coverings of $\P^1$) \label{Riemann existence} Let $\Delta=\{t_1,\dots, t_r\}\subset \mathbb{P}^1$ be a finite set and $G$ a finite group. Furthermore, for each $t_i$, let $C_{t_i}$ be a (non-trivial) conjugacy class of $G$. Then there exists a finite Galois covering of $\mathbb{P}^1\setminus\Delta$ with ramification given by the conjugacy classes $C_{t_i}$ if and only if there exist generators $g_1,\dots, g_r$ of $G$ such that $g_1\cdots g_r=1$ with $g_i\in C_{t_i}$.
\end{theorem}
\begin{proof}
For a proof of this theorem, see \cite{Vo}. 
\end{proof}
Let $t_k\in Y$ be a branch point of $f:C\to Y$. Without loss of generality we may assume $t_k=0$ in a local chart around the point (which is isomorphic to an open subset of $\C$). Choose a small disk centered in $t_k=0$. Let $e_k$ be the ramification index of $f$ above $0$.  Near the branch point $0$, the covering $f$ looks like the root map $z\mapsto z^{1/e_k}$. This means that we can lift a closed path around $0$ to a path starting from $(z,z^{1/e_k})$ and ending at the point $(z,\exp(\frac{2\pi i}{e_k})z^{1/e_k})$.\par

\subsection{Sheaves and Galois covers}

\subsubsection{{\bf $G$-Sheaves}}\label{g-sheaf}
 Most of the notation and presentation of this subsection comes from \cite{CLP}, \S 5.2. Let $Y$ is a complex smooth algebraic variety on which a finite group $G$ acts. The sheaf of $\sO_Y$-algebras $\sO_Y[G]$ is defined in a similar way as the regular representation $\C[G]$ of $G$.  A sheaf $\sF$ of $\sO_Y$-modules is called a \emph{$G$-sheaf}, if it has a structure of sheaf of $\sO_Y[G]$-modules. If moreover $\sF$  is a vector bundle then its fibers carry a linear $G$-action and so we can see $\sF$ as a family of representations of $G$ parametrized by $Y$. Recall that for any representation $\rho:G\to\gl(V)$ of $G$, its \emph{canonical decomposition } is the unique decomposition
\begin{equation}\label{can decomp}
V=V_1\oplus\ldots\oplus V_N,
\end{equation}
defined as follows. Let $W_1,\ldots, W_N$ be the different irreducible representations of $G$. Then each $V_i$ is the direct sum of all the irreducible
representations of $G$ in $V$ that are isomorphic to $W_i$. If $\chi_1,\ldots, \chi_N$ are the characters of $W_1,\ldots, W_N$, and $n_i=\dim W_i$, then 
\begin{equation}\label{projector}
p_i=\frac{n}{|G|}\displaystyle\sum_{g\in G}\overline{\chi_i(g)}\rho(g)\in\End(V),
\end{equation}
is the projection of $V$ onto $V_i$ for any $i=1,\ldots, N$, where $\overline{\chi_i(g)}$ is the complex-conjugate of $\chi_i(g)$.\par Let now $\sF$ be a locally free $G$-sheaf on $Y$ with action $\rho:G\to\gl_{\sO_Y}(\sF)$. Via ~\ref{projector}, we define an endomorphism $p_i\in\End_{\sO_Y}(\sF)$, for any $i=1,\ldots, N$. Setting $\deg\sF_i=\Ii(p_i)$, we have the following decomposition:
\begin{equation}\label{coh decomp}
\sF=\sF_1\oplus\ldots\oplus \sF_N,
\end{equation}
for any $i=1,\ldots, N$, $\sF_i$ is called the \emph{eigensheaf} of $\sF$ corresponding to the irreducible representation $\rho_i$ with character $\chi_i$. Notice that $\sF_i$ is a vector sub-bundle of $\sF$ for any $i$. In particular, when $\sF=\pi_*\sO_X, \pi:X\to Y$ is a flat $G$-cover, we have that $\pi_*\sO_X$ is a locally free sheaf of $\sO_Y[G]$-modules of rank one, i.e. the fibres of $\pi_*\sO_X$ are isomorphic to $\C[G]$ as $G$-representations. Indeed, the previous procedure gives the decomposition $\pi_*\sO_X=(\pi_*\sO_X)_1\oplus\ldots\oplus (\pi_*\sO_X)_N$ with $(\pi_*\sO_X)_i\subset \pi_*\sO_X$ a sub-bundle, for any $i$. By construction, the fibres of $(\pi_*\sO_X)_i$ are isomorphic to each other as $G$-representations. So, it is enough to consider the restriction of $\pi_*\sO_X$ on the complement of $Y\setminus\Delta$. \par Of course in this paper we are only interested in the case where $Y$ is a complex smooth algebraic curve. Note that if $f:C\to Y$ is a $G$-cover of curves,  then $f_*\omega_C$ carries a structure of $G$-sheaf, see \S ~\ref{g-sheaf}. As such, it has a decomposition into $G$-eignsapces ${(f_*\omega_C)}_{\rho_j}$. Also, if $\pi:\sC\to\sY$ is a family of $G$-covers of a curve $Y$, then the sheaf $R^1\pi_*(\C)$ has the structure of a $G$-sheaf with eigensheaves $\sL_{\rho_j}$. If $H^1(\sC_q,\C)_{\rho_j}=(\sL_{\rho_j})_q$ satisfies that $h^{1,0}(\sC_q)_{\rho_j}=n_1$ and $h^{0,1}(\sC_q)_{\rho_j}=n_2$, then we say that the eigenspace $\sL_{\rho_j}$ is of type $(n_1, n_2)$. 

\subsubsection{{\bf The local system corresponding to a Galois cover}}
In this subsection, we study the sheaves under $G$-Galois covers in more details. Let us begin with a well-known result. 
\begin{lemma} \label{local system}
Let $U$ be an arcwise connected and locally simply connected topological space and $x\in U$. Then there is a bijection between complex local systems of rank $n$ over $U$ and (monodromy) representations $\theta:\pi_1(U)\to \gl_n(\C)$. Under this bijection, irreducible local systems correspond to irreducible representations. 
\end{lemma} 
We apply Lemma \ref{local system} in the case $U=Y\setminus \Delta$ as above, where $\Delta\subset Y$ is a finite subset. Let us describe the local system $f_*(\C)\mid_{U}$. Let $D\subset U$ be an arbitrary disc. The preimage $f^{-1}(D)$ is the disjoint union of $n=|G|$ discs. We label these discs by the elements of $G$ and write  $f^{-1}(D)=\bigsqcup\limits_{g\in G}D_g$ where each $D_g$ is a disc and the Galois group $G$ acts as $h\cdotp D_g=D_{h\cdotp g}$. This action corresponds to the regular representation $\rho:G\to \gl(\C[G])\cong \gl_n(\C)$. Hence the vector space $f_*(\C)\mid_{U}(D)$ has the structure of a $\C[G]$-module (equivalently a $G$-representation). Theorem \ref{RET} amounts to give an epimorphism $\Phi:\pi_1(U)\to G$ such that $\Delta=\br(f)$. As the construction of this map is useful in relating the $\C[G]$-module to the representation $\theta$, we explain this in more details here: Let $V=f^{-1}(U)$ as before, so that $f:V\to U$ is a topological (unramified) Galois covering of degree $n$. Choose $x\in V$ and let $f(x)=y\in U$. Consider the fundamental group $\pi_1(U, y)$ and let $[\gamma]\in\pi_1(U, y)$ be the class of a loop $\gamma$ in $U$. By the lifting property of Hausdorff spaces (see e.g., \cite{F}), $\gamma$ lifts to a (unique) path $f^*(\gamma)$ in $V$ with starting point $f^*(\gamma)(0)=x$ and ending point $f^*(\gamma)(1)=g\cdotp x$ for some $g\in \deck(C/Y)\cong G$. The epimorphism $\Phi$ associates $g$ with the class $[\gamma]\in\pi_1(U, y)$. On the other hand, the construction of the representation $\theta$ (see e.g., \cite{Sz}) corresponding to the local system $f_*(\C)$, shows that $\theta$ factors through the action of $G$, i.e., the regular representation of $G$ described above. So we have the following commutative diagram:
\begin{equation} \label{commutative repres}
\begin{tikzcd}
& G \arrow{dr}{\rho} \\
\pi_1(U) \arrow{ur}{\Phi} \arrow{rr}{\theta} && \gl_n(\C)
\end{tikzcd}
\end{equation}
The $\C[G]$-module $f_*(\C)\mid_{U}(D)$ described above, decomposes into a direct sum of irreducible $\C[G]$-modules. Let $\rho_1,\dots, \rho_s$ be the irreducible representations of $G$, with $\deg \rho_i=n_i.$ Then $f_*(\C)\mid_{U}(D)=\bigoplus\limits_{i=1}^{s}L_i$, where $L_i$ is an irreducible $G$-representation of degree $n_i$. Since $D$ is arbitrary, this argument globalizes and we obtain 
\begin{equation} \label{local system decomposition}
f_*(\C)\mid_{U}=\bigoplus\limits_{i=1}^{s}\L_{\rho_i},
\end{equation}
where $\L_{\rho_i}$ is the eigensheaf corresponding to the irreducible representation $\rho_i.$ It is a local system of rank $n_i$. Let $\theta:\pi_1(U)\to \gl_n(\C)$ be the representation corresponding to the local system $f_*(\C)\mid_{U}$ by Lemma \ref{local system}. The representation $\rho$ decomposes into irreducible representations
\begin{equation}
\theta=\theta_{\rho_1}\oplus\dots\oplus \theta_{\rho_s}:\pi_1(U)\to \prod\limits_{i=1}^s\gl_{n_i}(\C) \text{ where }\\
\theta_{\rho_i}:\pi_1(U)\to \gl_{n_i}(\C)
\end{equation}
is the monodromy representation of the eigenspace $\L_{\rho_i}$ as in \ref{local system decomposition}. By \ref{commutative repres}, each $\theta_{\rho_i}$ is also induced by a $\rho_i:G\to \gl_{n_i}(\C)$. Consider the branch point $t_k$ and denote the inertia group of the ramification point $x_{j_k}\in f^{-1}(t_k)$ with $H_{j_k}=\langle h_{j_k}\rangle$. Let $\gamma_k$ be a small loop around $t_k$ such that $\{\gamma_k\}\subset U$. So the possible lifts of this loop are paths with starting point $\gamma_k^g(0)=y_g\in D_g$ and ending point $\gamma_k^g(1)=y_{h_{j_k}\cdotp g}\in D_{(h_{j_k}\cdotp g)}$ (note that the discs are labeled by elements $g\in G$). Therefore we have
\begin{theorem} \label{local system repres}
The local system $f_*(\C)\mid_{U}$ associated to the Galois cover $f:C\to Y$ is given by the monodromy representation $\theta:\pi_1(U)\to \gl_n(\C)$ as follows
\begin{equation}
\theta(\gamma_k)(\begin{bmatrix}
x_1 \\
x_{2} \\
\vdots \\
x_s \\
\end{bmatrix})=\begin{bmatrix}
\rho_1(h_k)(x_1) & & &0\\
& \rho_2(h_k)(x_2)& &\\
& & \ddots &\\
0& & & \rho_s(h_k)(x_s)
\end{bmatrix},
\end{equation}
where $\gamma_k$ is a closed path around the branch point $t_k$ and $x_j$ is a local section of $\L_{\rho_j}$. 
\end{theorem}
\subsection*{Dual and conjugate local systems}
Let $L$ be a local system of $\mathbb{C}$-vector spaces over a topological space $X$. To this local system, one can associate another local system called \emph{the dual local system} denoted by $L^{\vee}$ given by $L^{\vee}(\sU):=L(\sU)^{\vee}=\Hom(L(\sU),\mathbb{C})$ for any open subset $\sU$. The local monodromy of $L^{\vee}$ is given by the dual map $\theta^{\vee}$.

\subsection{The quotient Galois cover} \label{quotient Galois cover}
Let $N\lhd G$ be a normal subgroup of $G$. The Galois cover $f:C\to Y$ factors through a Galois cover $\overline{f}:\overline{C}=C/N\to Y$ with Galois group
$\overline{G}=G/N$. Let $\Delta=\br(f)$ as before  and $\overline{\Delta}=\br(\overline{f})$. In order to compute the monodromy representation of the local system $\overline{f}_*(\mathbb{C})\mid_{Y\setminus \overline{\Delta}}$, we will need the following.
\begin{lemma} \label{irr rep of quotient}
Over any field the irreducible representations of the quotient group $G/N$ are just the irreducible representations of $G$ whose restriction to $N$ is trivial. Moreover, let $\mu:\overline{G}\to \gl_{m}(\mathbb{C})$ be an irreducible representation of $\overline{G}$ and $q:G\to \overline{G}$ be the natural quotient map. Then $\mu\circ q:G\to \gl_{m}(\mathbb{C})$ is an irreducible representation of $G$. In particular, the (multi)set of degrees of irreducible representations of $\overline{G}$ is a sub-(multi)set of that of degrees of irreducible representations of $G$.
\end{lemma}
The following lemma will be useful to understand the ramification behavior of the quotient
cover.
\begin{lemma} \label{ramification of quotient}
Let $f:C\to Y$ be a Galois $G$-cover of degree $|G|=n$ and $N\lhd G$ a normal subgroup. Consider the following diagram in which $\overline{f}:\overline{C}\to Y$ is the quotient cover with the notations as above. 
\[\begin{tikzcd}
C \arrow{rr}{f} \arrow[swap]{dr}{\varphi} & & Y \\
& \overline{C} \arrow[swap]{ur}{\overline{f}}
\end{tikzcd}\]
Then
\begin{enumerate}
\item  $\overline{\Delta}\subseteq \Delta$, i.e., the branch points of $\overline{f}$ are also branch points of $f$. \item $\ram(\varphi)\subseteq \ram(f)$, i.e., the ramification points of $\varphi:C\to\overline{C}=C/N$ are also ramification points of $f$. \item $\ram(\overline{f})\subseteq \varphi(\ram(f))$, i.e., every ramification point of $\overline{f}$ is the image of a ramification point of $f$ under $\varphi$ and the inertia group of a ramification point $\overline{p}=\varphi(p)$ of $\overline{f}$ is (isomorphic to) $G_p/N\cap G_p$.
\end{enumerate}
\end{lemma}
\begin{proof}
The first statement is true simply for degree reasons: If $b\in Y$ is a branch point of $\overline{f}$, then $|\overline{f}^{-1}(b)|<\frac{|G|}{|N|}=\deg \overline{f}$. Since $f=\overline{f}\circ \varphi$, it follows that $|f^{-1}(b)|<\deg \overline{f}\cdotp |N|=|G|=\deg f$.\par For the second assertion, let $p\in C$ be a ramification point of
$\varphi$. Then the stabilizer of $p$ in $N$ and hence in $G$ is non-trivial. For the statements in (3), let $\overline{p}=\varphi(p)$ be a ramification point of
$\overline{f}$. Then by definition, $b:=\overline{f}(\overline{p})$ is branch a point of
$\overline{f}$. By (1) this is also a branch point of $f$ and since the covers are Galois, every point in $f^{-1}(b)$ is a ramification point of $f$. It suffices to show that $p\in f^{-1}(b)$ which is evident by $f=\overline{f}\circ\varphi$. Now let $\overline{G}_{\overline{p}}=\langle\overline{g}\rangle$ be the inertia group of $\overline{p}$ in $\overline{G}$, where $\overline{g}=gN$ with $g\in G$. Hence $\overline{g}(\overline{p})=\overline{p}$ or equivalently $\varphi(g(p))=\varphi(p)$. We claim that $\overline{G}_{\overline{p}}= NG_p/N\cong G_p/N\cap G_p$. The inclusion $NG_p/N\subset \overline{G}_{\overline{p}}$ is clear. We show the other inclusion. By \cite{Mir}, Proposition 3.3(a), there exists an open set $\overline{U}_{\overline{p}}$ in $\overline{C}$ which is invariant under $\overline{G}_{\overline{p}}$. Let $U_p$
be the corresponding open neighborhood of $p$ in $C$. Since $\varphi$ is a holomorphic map between compact Riemann surfaces, it is also open, so shrinking if necessary, we may assume that $\varphi(U_p)\subset \overline{U}_{\overline{p}}$. Furthermore, the condition $\varphi(g(p))=\varphi(p)$ implies that $U_p\cap g\cdotp U_p\neq \emptyset$. Now \cite{Mir}, Proposition 3.3(c) implies that $g\in G_p$.
\end{proof}
We remark that all of the inclusions in Lemma \ref{ramification of quotient} are in general proper.\par
Suppose without loss of generality that $\rho_1,\dots, \rho_l$ are the irreducible representations of $G$ that have the normal subgroup $N$ contained in their kernels. Using Lemma \ref{irr rep of quotient}, it follows that the irreducible representations of $\overline{G}$ are precisely $\overline{\rho}_1,\dots, \overline{\rho}_l$ where $\overline{\rho}_i:\overline{G}\to \gl_{n_i}(\mathbb{C})$ is the reduction of $\rho_i$ to $\overline{G}$ for $1\leq i\leq l$. By Lemma \ref{ramification of quotient}, a ramification point of $f$ is a ramification point of the quotient cover $\overline{f}$ if and only if the inertia group $H_p\not\subset N$, in which case the ramification index of $\overline{p}\in \overline{C}$ is $|H_p|/|N\cap H_p|$. Let $\overline{H}_p=H_p/N\cap H_p=\langle\overline{h}_p \rangle$. The cover $f:C\to Y$ factors through a Galois cover $\overline{f}:\overline{C}\to Y$ whose associated local system is given by
\begin{proposition}
The local system $\overline{f}_*(\mathbb{C})\mid_{Y\setminus \overline{\Delta}}$ associated to the Galois cover $\overline{f}:\overline{C}\to Y$ is given by the monodromy representation $\overline{\theta}:\pi_1(Y\setminus\overline{\Delta})\to \gl_{\overline{n}}(\C)$ as follows
\begin{equation}
\overline{\theta}(\gamma_k)(\begin{bmatrix}
x_1 \\
x_{2} \\
\vdots \\
x_l \\
\end{bmatrix})=\begin{bmatrix}
\overline{\rho}_1(\overline{h}_k)(x_1) & & &0\\
& \overline{\rho}_2(\overline{h}_k)(x_2)& &\\
& & \ddots &\\
0& & & \overline{\rho}_l(\overline{h}_k)(x_l)
\end{bmatrix},
\end{equation}
where $\gamma_k$ is a closed path around the branch point $t_k$ and $x_j$ is a local section of $\L_{\rho_j}$.  
\end{proposition}

\subsection{$\L_{\rho}$-valued paths} \label{valuedpath}
Let $\gamma:[0,1]\to Y$ be a path in $Y$. We may assume that $\gamma(0,1)\subset \sU\subset U$, where $\sU$ is a simply connected open subset of $U=Y\setminus \Delta$. If this is not the case, decompose $\gamma$ into such paths. Consider the Galois cover $f:C\to Y$. The curve $\gamma$ has $n=|G|$ liftings $\{\gamma_{g_1},\dots, \gamma_{g_n}\}$ to $C$ (where $G=\{g_1,\dots, g_n\}$). Let $\mathcal{C}$ be the free abelian group generated by the paths on $C$. Then
$\mathcal{C}\otimes_{\mathbb{Z}}\mathbb{C}$ is the vector space of $\mathbb{C}$-valued paths on $C$. Recall the decomposition $f_*(\mathbb{C})\mid_{U}=\oplus\mathbb{L}_{\rho_j}$. Let $c_j\in\mathbb{L}_{\rho_j}(\sU)$ be a section over $\sU$. Since $f$ is unramified over $\sU$ and $\sU\subset U$ is simply connected by assumption, $f^{-1}(\sU)=\bigsqcup\limits_{i=1}^{n}\sV_i$ for open sets $\sV_i\subset C$. Restricted to each $\sV_i$, $f$ induces a homeomorphism $\sV_i\cong\sU$. Moreover, notice that
$f_*(\mathbb{C})\mid_{U}(\sU)=\mathbb{C}(f^{-1}(\sU))=\mathbb{C}(\bigsqcup\limits_{i=1}^{n}\sV_i)=\bigoplus\limits_{i=1}^{n}\mathbb{C}$. Hence for each open sheet $\sV_i$ over $\sU$, there are local sections $c_{j_i}\in\mathbb{C}(\sV_i)=\mathbb{C}$ that give rise to the section $c_j\in\mathbb{L}_{\rho_j}(\sU)\subset f_*(\mathbb{C})\mid_{U}(\sU)$. A \emph{$\mathbb{L}_{\rho_j}$-valued path} $\widetilde{\gamma}$ is then given by the combination
\begin{equation}\label{valued path}
c_j\cdotp \gamma:=\widetilde{\gamma}=\sum_{i=1}^{n}
c_{j_i}\gamma_{g_i}
\end{equation}
In general, $\mathbb{L}_{\rho_j}$ may extend over a larger subset $Y\setminus \Delta_j\supset Y\setminus\Delta$, i.e., it can happen that the monodromy of
$\mathbb{L}_{\rho_j}$ at some $x\in \Delta$ is trivial. We therefore denote by $\Delta_j\subset \Delta$ the largest subset such that the local system $\mathbb{L}_{\rho_j}$ \emph{does not} extend over any larger subset than $U_j=Y\setminus\Delta_j$. So $\Delta\setminus \Delta_j$ is the set of all points
at which the monodromy of $\mathbb{L}_{\rho_j}$ is trivial.\par Note that the Galois group acts on $\{\gamma_{g_1},\dots, \gamma_{g_n}\}$ by : $g\cdotp \gamma_{g_i}=\gamma_{g\cdotp g_i}$. Under this action, one sees that $G$ acts on the $\mathbb{L}_{\rho_j}$-valued path $\widetilde{\gamma}$ in \ref{valued path} via the representation $\rho_j$. \par Now
suppose $\gamma^{(k)}$ is an oriented path from the branch point $t_k$ to the branch point $t_{k+1}$ that does not go through any other branch points of $f$. In fact, for simplicity, we may take $\gamma^{(k)}$ to be a straight line. Let $c^{(k)}_{j}$ be a non-zero local section of $\mathbb{L}_{\rho_j}$ over the open set $\sU=\gamma^{(k)}((0,1))$ given as above by the sections $\{c^{(k)}_{j_i}\}_{i=1}^{n}$ on the sheets over $\sU$. Then $G$ acts on the $\mathbb{L}_{\rho_j}$-valued path $c^{(k)}_j\cdotp \gamma^{(k)}$ by $\rho_j$ and hence its class $[c^{(k)}_j\cdotp \gamma^{(k)}]$ in the homology group $H_1(C,\mathbb{C})$ belongs to the eigenspace $H_1(C,\mathbb{C})_{\rho_j}$.
\begin{theorem}
Let $Y=\P^1$ and $c^{(k)}_{ij}, i=1,\dots, l_j$ be a $\mathbb{C}$-basis for $\mathbb{L}_{\rho_j}(\sU)$. Then the classes $[c^{(k)}_{ij}\cdotp\gamma^{(k)}], i=1,\dots, l_j, k=1,\dots, n_j$ form a
$\mathbb{C}$-basis for the space $H_1(C,\mathbb{C})_{\rho_j}$. Here $n_j=|\Delta_j|$.
\end{theorem}
\begin{proof}
In fact, $\P^1$ has the structure of a finite cell complex by taking the singletons $\{t_1,\dots, t_{n_j}\}$ as as $0$-cells, the intervals $\{\gamma_1,\dots, \gamma_{n_j}\}$ as $1$-cells and $Y\setminus \bigcup\limits_{i=1}^{n_j}\gamma_i$ as $2$-cell. By considering the connected components of the preimages of cells in $C$, we see that $C$ also obtains the structure of a finite cell complex. There is a cellular chain complex of $C$ 
\[0\to C_2\to C_1\to C_0\to 0 \]
equipped with a $G$-action. $C_1$ is the free $\Z[G]$-module generated by $\gamma_1,\ldots, \gamma_{n_j}$. Let $(G:H_k)$ be the set of representatives of the cosets of the centralizer $C(H_k)$ of $H_k$  or equivalently the number of distinct conjugates of the subgroup  $H_k$ in $G$. Then $C_0\cong \bigoplus\limits_{k=1}^{n_j}\Z[(G:H_k)]$, i.e., the free $\Z$-module on the set $(G:H_k)$. Similarly, $C_2\cong\Z[(G:H_{\infty})]$. In both cases, $(C_0)_{\rho_j}=(C_2)_{\rho_j}=0$, which proves the claim.   
\end{proof}
\subsection{Collision of points}
Let $\L_{\rho_j}$ be the eigenspace in the cohomology of a fiber $C=\sC_q$. Let $|\Delta_j|=n_j$, where $\Delta_j$ is the subset as introduced in ~\ref{valuedpath}. This eigenspace is therefore defined over a subset $\{a_1,\ldots, a_{n_j}\}$ of $\Delta$. Consider now the partition $P=\{\{a_1\},\ldots, \{a_{n_{j-1}}, a_{n_j}\}\}$ of $\Delta_j$. In what follows, we set $b=\{a_{n_{j-1}}, a_{n_j}\}$. Let $\psi_P: P\to Y$ be some embedding and the local system $\L(P)_{j}$ on $Y\setminus\psi_P(P)$ with the local monodromy given by an epimorphism $\Phi^{\prime}:\Gamma_{g^{\prime},{n_j-1}}\to G^{\prime}$ as in Definition ~\ref{datum} such that $\Phi^{\prime}(\gamma_i)=\Phi(\gamma_i)$ for $i\leq n_j-2$ and $\Phi^{\prime}(\gamma_b)=\Phi(\gamma_{n_j-1}\gamma_{n_j})$. Here $G^{\prime}$ is the subgroup of $G$ generated by $\gamma_1,\ldots, \gamma_{n_j-2}, \gamma_b$. Theorem ~\ref{RET} then guarantees that there exists a family $\pi(P):\sC(P)\to \sT_{n_j-1}$ of $G^{\prime}$-covers of $Y$. We denote the eigenspace of the higher direct image sheaf $R^1\pi(P)_*(\C)$ with respect to the character given by 1 by $\sL(P)_{j}$. This eigenspace does not neccessarily come from an eigenspace of an irreducible family of $G$-covers by collision of points, even when $G$ is cyclic. The problem is that the resulting family obtained by collision of points is not in general irreducible.
\section{families of Galois covers of $\P^1$}
In this section, we discuss families of Galois covers of the Riemann sphere. Let $\Delta=\{t_1,\ldots, t_r\}$ be a set of points in $\P^1$ and $U=\P^1\setminus\Delta$. We fix a presentation $\Gamma_r$ for $\pi_1(U)$ as in ~\ref{braid}. Then we have an epimorphism of groups $\Phi:\pi_1(U)\to\Gamma_r$. Suppose that $\gamma_i$ is a generator of $\pi_1(U)$ and that $\Phi(\gamma_i)$ has order $m_i$. By the Riemann's existence theorem, any epimorphism of groups as above and any $r$-tuple of integers $(m_1,\ldots, m_r)$ gives rise to a Galois cover of $\P^1$ with $r$ branch points. This can also be done in families: To any datum as above, we can associate a family of $G$-Galois covers of $\P^1$. In this case, let
\[\mathcal{T}_r=(\P^1)^{r+3}\setminus\{t_i\neq t_j\}\]
Then we get a family $f:\sC\to\sT_r$ of $G$-branched covers of $\P^1$ branched over $r$ points. Note that  the space $\mathcal{T}_r$ is the configuration space of $r+3$ points. The fundamental group of $\mathcal{T}_r$ can be identified with the \emph{braid group} on $r+3$ strands in $\P^1$ which is given by the braids leaving the strands invariant. An element of this group is for example given by the Dehn twist $T_{k_1, k_2}$ with $1\leq k_1\leq k_2\leq r+3$. The
Dehn twist $T_{k_1, k_2}$ is given by leaving the branch point $t_{k_2}$ run counterclockwise around $t_{k_1}$. Now consider a fiber $C=\sC_q$ of $\sC$. The fiber $C$ is by construction a $G$-Galois cover of $\P^1$ already described above. The Dehn twist $T_{l, l+1}$ leaves $\delta_k$ invariant for all $k$ with $|l-k|>1$. If $G$ is abelian, then the matrix $M_{l, l+1}$ of $T_{l, l+1}$ has only the entries  $a_{l,l}, a_{(l-1),l}, a_{(l+1),l}$ non-zero and also $a_{k,k}=1$ for $k\neq l$. All other entries are equal to 0.  If $G$ is abelian, we therefore have
\begin{lemma}
$\det T_{l, l+1}=a_{l,l}$. Also the eigenvalues of $T_{l, l+1}$ are given by 1 and $a_{l,l}$. Furthermore the matrix induced by the Dehn twist $T_{l, l+1}$ is diagonalizable.
\end{lemma}
\begin{proof}
This follows directly by the description of the matrix $T_{l, l+1}\in\gl (H^{0}(\Omega^1_C))$. In particular, we get once again $a_{l,l}\neq 0$. Also, the minimal polynomial of the matrix of $T_{l, l+1}$ is $(x-1)(x-a_{l,l})$.
\end{proof}
In the particular case that $a_{l,l}=1$, or equivalently, $T_{l, l+1}\in\SL(H^{0}(\Omega^1_C))$, we have the following
\begin{proposition}\label{monod diag}
If the eigenspace $\sL_{\rho_j}$ is of the type $(1,1)$ and $T_{1,2}\in\SL(H^{0}(\Omega^1_C))$, then the monodromy group of $\sL_{\rho_j}$ is infinite.
\end{proposition}
\begin{proof}
If the Dehn twist of the fiber $C:=\sC_q$ has the asserted form for some $q\in\sT_r$, then the Dehn twist $T_{1,2}$ is a unipotent and upper-tringular matrix by the above description and therefore the claim is proved. 
\end{proof}

\subsection{Dihedral covers of curves}
Throughout this paper, we will fix the following presentation of the dihedral group: 
\begin{equation} \label{dihedral presentation}
D_n\coloneqq \langle a,b| a^n=b^2=1, ab=ba^{-1}\rangle.
\end{equation} 
$D_n$ sits therefore in the following exact sequence.
\begin{equation} \label{dihedral extension}
0\to\Z/n\Z\cong\langle a\rangle\to D_n\to\langle b\rangle\cong\Z/2\Z\to 0
\end{equation}
Now let $X$ be a smooth algebraic curve over $\C$ with $D_n\subset \aut(X)$ and $Y$ a smooth complex algebraic curve such that $X/D_n=Y$ and the cover $\pi: X\to Y$ is Galois. We are interested in the quotient map $X\to X/D_n\coloneqq Y$. Let $A=\langle a\rangle$ and $N=\langle b\rangle$. The factorization \ref{dihedral extension} gives rise to a factorization $\pi \colon X\xrightarrow{p}Z\xrightarrow{q} Y$ where $p,q$ are the corresponding intermediate cyclic Galois covers, i.e., $p\colon X\to Z=X/A$ is a cyclic Galois covering with Galois group $A$ and $q\colon Z\to Y=Z/N$ is a cyclic Galois covering with Galois group $N$. Therefore to study the Galois covering $\pi \colon X\to Y$, it is helpful to study these intermediate cyclic coverings. 
On the other hand, a dihedral Galois covering $\pi:X\to Y$ gives rise to a $D_n$-field extension $\C(Y)\subset \C(X)$. The following result describes the diehdral field extensions of the function field $\C(Y)$. For a proof, see \cite{CP}, Proposition 5.2.
\begin{proposition}\label{dihedral field extension}
Let $\C(Y)\subset \C(X)$ be a $D_n$- extension. There exist $G,F\in \C(Y)$ and $x\in \C(X)$ with
\begin{equation}\label{dihedral equation}
x^{2n}-2Gx^n+F^n=0,
\end{equation}
such that $\C(X)=\C(Y)(x)$ and the $D_n$-action is given by $\sigma(x)=\xi_{n}x$ and $\tau(x)=\frac{F}{x}$, where $\xi_{n}$ is a primitive $n$-th root of unity in $\C$.\par 
Conversely, given $G,F\in \C(Y)$ such that $x^{2n}-2Gx^n+F^n$ is irreducible in $\C(Y)[x]$ then $\frac{\C(Y)[x]}{(x^{2n}-2Gx^n+F^n)}$ is a $D_n$-Galois extension of $\C(Y)$. Hence the normalization of $Y$ in $\frac{\C(Y)[x]}{(x^{2n}-2Gx^n+F^n)}$ is a $D_n$-covering of $Y$. 
\end{proposition}
Recall that $D_n$ has irreducible degree $2$ representations $\rho_h$. The representation $\rho_h$ can be given by
\begin{equation}\label{irr repres}
\rho_h(a^k)=\begin{bmatrix}
\xi_n^{hk} & 0\\
0& \xi_n^{-hk}\\
\end{bmatrix}, \hspace{2cm} \rho_h(a^kb)=\begin{bmatrix}
0&\xi_n^{hk}\\
\xi_n^{-hk}& 0\\
\end{bmatrix}
\end{equation}
where $\xi_n$ is a primitive $n$-th root of unity. Also, if $n$ is odd, then $1\leq h\leq \frac{n-1}{2}$ and  if $n$ is even, then $1\leq h\leq \frac{n}{2}-1$. For a degree 2 irreducible representation of $D_n$ as above, we denote by $\C(X)^{\rho_h}$ the eigenspace of the action of $D_n$ with resptect to this representation. By Proposition \ref{dihedral field extension}, as a vector space over $\C(z)$, 
\begin{equation}\label{FF eigenspace}
\C(X)^{\rho_h}=\langle x^h, F^hx^{-h}\rangle\oplus\langle x^{n-h}, F^{n-h}x^{-(n-h)}\rangle
\end{equation}
One can also classify the $D_n$-covers of algebraic curves and also determines the eigenspaces of group action on $H^0(C,\omega)$.   
\begin{theorem} [Structure of $D_n$-covers] \label{structure of dihedral} A $D_n$-Galois cover $f \colon C\to Y$ is determined by the following data:
\begin{enumerate}
\item A line bundle $\sL$ and an effective reduced divisor $B_q$ on $Y$ such that $\sL^{\otimes 2}\equiv \sO_Y(-B_q)$.
\item Reduced effective Weil divisors $D_1,\dots, D_n$ on $Z=\spec(\sO_Y\oplus \sL)$ such that $\overline{\tau}(D_{k})=D_{n-k}$ for $1\leq k\leq \lfloor \frac{n}{2}\rfloor$.\item Rank one reflexive sheaves $\sF_1,\dots, \sF_n$ on $Z$, flat over $\sO_Y$ such that $\overline{\tau}^*(\mathcal{F}_{l})=\sF_{n-l}$ on $Z$ for $1\leq l\leq \lfloor \frac{n}{2}\rfloor$ such that the linear equivalence \ref{fundamentalrel2} holds.
\end{enumerate}
\end{theorem}
It follows from Theorem \ref{structure of dihedral} that 
\begin{proposition} \label{metacyclic decomposition}
Let $Y$ be a smooth variety and $f:C\to Y$ a flat $D_n$-cover and let $p\colon C\to Z$ and $q\colon Z\to Y$ be the intermediate coverings of degrees $n$ and $2$ respectively. Then
\begin{equation}
\pi_*\mathcal{O}_X=\bigoplus\limits_{i=1}^{\nu} (\pi_*\mathcal{O}_X)_{\rho_h},
\end{equation}
where $(\pi_*\mathcal{O}_X)_{\rho_h}$ 
\begin{equation}
(\pi_*\mathcal{O}_X)_{\rho_h}=U_{h}\oplus U_{n-h},
\end{equation} 
for $h=1,\dots, \frac{n-1}{2}$ if $n$ is odd and $h=1,\dots, \frac{n}{2}-1$ if $n$ is even and $U_j=q_*(\mathcal{F}_j).$ 
\end{proposition}
In this paper, we only apply this theorem to the case $Y=\P^1_{\C}$. Next, let us describe the subgroups of the dihedral group $D_n$. 
\begin{remark} \label{subgroups of dihedral}
The subgroups of the dihedral group $D_n$ are as follows:
\begin{enumerate}
\item The cyclic subgroups which are precisely the subgroups $\langle a^i\rangle, i|n$ of index $2i$ and the subgroups $\langle a^ib\rangle$ of order $2$ (reflections). \item The non-cyclic subgroups are of the form $\langle a^j, a^ib\rangle$ of index $j$, where $j|n$ (but $j\neq n$) and $0\leq i\leq j-1.$
\end{enumerate}
The subgroups of type (2) in Remark \ref{subgroups of dihedral} are isomorphic to the dihedral group $D_{\frac{n}{j}}$. In particular, subgroups of the dihedral group are either cyclic or dihedral. Furthermore, if $n$ is odd, then the only cyclic subgroups of $D_n$ of order two are precisely the subgroups $\langle a^ib\rangle$, whereas, if $n$ is even, there is one additional such subgroup, namely, $\langle a^{\frac{n}{2}}\rangle$.
\end{remark}
In order to determine the ramification behavior of the dihedral covers, we need also to investigate the conjugacy classes of the dihedral group $D_n$. This depends on the parity of $n$. 
\begin{remark}\label{conjugacy of dihedral}
If $n$ is odd, then the conjugacy classes are the identity $\{1\}$, the sets
$\{a^i,a^{-i}\}$ with $i=1,\dots, \frac{n-1}{2}$ and the set $D_n\setminus \langle a\rangle$ (i.e., the set $D_n\setminus \langle a\rangle=\{a^ib\mid 0\leq i\leq n-1\}$ is a single conjugacy class, or in other words, all of the reflections are conjugate to each other). In particular, the number of the conjugacy classes is equal to $\frac{n+3}{2}$. If $n$ is even, the conjugacy classes are the identity $\{1\}$, the sets $\{a^i,a^{-i}\}$ with $i=1,\dots, \frac{n}{2}$ and there are two more classes, namely, $\{a^ib\mid
i\text{ odd}\}$ and $\{a^ib\mid i\text{ even}\}$. In particular, the number of the conjugacy classes is equal to $\frac{n+6}{2}$. It follows from this description of the conjugacy classes that if $n$ is odd, then two cyclic subgroups $\langle a^i\rangle$ and
$\langle c\rangle$ are conjugate if and only if $\langle a^i\rangle=\langle c\rangle$.
\end{remark}
Using Remarks \ref{subgroups of dihedral} and \ref{conjugacy of dihedral}, we obtain the following:
\begin{lemma} \label{conj ram}
Let $f:C\to Y$ be a $D_n$-cover of curves with $\br(f)=\{t_1,\dots, t_r\}$ such that the ramification index of $t_j$ is $e_j$. Then:
\begin{enumerate}
\item If $e_j\neq 2$, the inertia groups of all of the points $x_{j_i}\in f^{-1}(t_j)$ are equal. Consequently, it makes sense to say that the inertia group of $t_j$ is a cyclic group $H_j=\langle h_j\rangle$.\item If $e_j=2$ and $n$ is odd, then the
inertia group $H_{j_i}$ of every $x_{j_i}\in f^{-1}(t_j)$ is of the form $H_{j_i}=\langle a^{k_j}b\rangle$ (i.e., generated by a reflection). If $n$ is even, then either all of the inertia groups $H_{j_i}$ are equal to $\langle a^{\frac{n}{2}}\rangle$ or they are of the form $\langle a^{k_j}b\rangle$ where the $k_j$ are all odd or all even.
\end{enumerate}
\end{lemma}
The multiplicity of a given irreducible representation of $G$ in $H^0(C,\Omega^1_{C})$ can be computed by the following classical result of Chevalley and Weil in \cite{CW}.
\begin{theorem}[Chevalley-Weil] \label{Chevalley-Weil}
Let $f:C\to Y$ be a Galois covering of complex algebraic curves with Galois group $G$. Let $e_i$ be the ramification index corresponding to the branch point $t_i \in Y$ and let $g_i$ be an element of order $e_i$ in $G$ corresponding to $t_i$. Let $\rho$ be an irreducible representation of $G$ and $N_{i,\alpha}$ denote the number of eigenvalues of $\rho(g_i)$ that are of the form $\xi_{e_i}^{\alpha}$, where $\xi_{e_i}=\exp({2\pi\sqrt{-1}/e_i})$. Then the multiplicity of $\rho$ in $H^0(C,\Omega^1_{C})$ is given by
\begin{equation}
\mu_{\rho}=d_{\rho}(g_Y-1)+\sum\limits_{i=1}^{r}\sum\limits_{\alpha=0}^{e_i-1}N_{i,\alpha}\langle
-\frac{\alpha}{e_i}\rangle+\epsilon
\end{equation}
Here $\epsilon=1$ if $\rho$ is the trivial representation and $\epsilon=0$ otherwise. Also, note that we denote by $\langle q\rangle$ the fractional part of $q\in\Q$. 
\end{theorem}
By combining Theorem \ref{Chevalley-Weil} with the description ~\ref{irr repres} of irreducible $2$-dimensional representations $\rho_h$ of the dihedral group $D_n$, we obtain the following result. Note that we write $\mu_{h}$ for the multiplicity $\mu_{\rho_h}$. 
\begin{corollary} \label{Chevalley-Weil proj}
Let $f:C\to Y$ be a Galois covering of algebraic curves with Galois group $D_n$. Let $\{t_{1},\ldots,t_r\}$ be the branch points and let $\{t_{1},\ldots, t_{r-l}\}$ be the branch points such that $t_i$ has monodromy given by the element $a^{k_i}$ in $D_n$ for some $k_i$ (with the notation in Presentation ~\ref{dihedral presentation}) as in Proposition ~\ref{even}.  Then the multiplicity $\mu_{h}$ of the irreducible representation $\rho_h$ in $H^0(C,\Omega^1_{C})$ is given by
\begin{equation}
\mu_{h}=2(g_Y-1)+\frac{l}{2}+\sum\limits_{i=1}^{r-l}(\langle-\frac{h k_i}{n}\rangle+\langle\frac{h k_i}{n}\rangle).
\end{equation}
Consequently,
$h^0(C)_{\rho_h}=\dim_{\mathbb{C}}H^0(C,\Omega^1_{C})_{\rho_h}=2\mu_{h}$.
\end{corollary}

For a $D_n$-cover $f:C\to\P^1$, we get from Corollary ~\ref{Chevalley-Weil proj} the following about the multiplicty of degree 2 irreducible representation of a $D_n$.
\begin{remark} \label{CW pr}
It follows from Theorem ~\ref{Chevalley-Weil proj} that if $\{t_{1},\ldots, t_u\}$ is the subset of $\{t_{1},\ldots, t_{r-l}\}$ consisting of branch points for which $n\nmid h k_i$ ($1\leq i\leq u$). Then 
\[\mu_{h}=2(g_Y-1)+\frac{l}{2}+u.\]
\end{remark}

\subsubsection{Dihedral covers of $\P^1$}
In the case $Y=\P^1$, one can say more about the above results. In particular, using Lemma ~\ref{conj ram}, one obtains the following result about the ramification behavoir of dihedral covers of curves. We may assume that all of the points of ramification index 2 orrespond to the conjugacy class of the same reflection $a^kb$.
\begin{proposition} \label{even}
Let $f:C\to \mathbb{P}^1$ be a $D_n$-cover with $\br(f)=\{t_{1},\dots,t_r\}$ where the ramification index of $t_j$ is $e_j$. Suppose, without loss of generality, that
$\{t_{r-l+1},\dots,t_r\}$ is the set of branch points with ramification index 2 that correspond to the conjugacy class of the reflection $a^kb$ as mentioned above. Let the rest of the branch points correspond to conjugcy classes $a^{d_i}$ with $d_i|n$ by Remark \ref{subgroups of dihedral}. Then the cover $f$ is not ramified over the infinity if and only if
\begin{enumerate}
\item $l$ is an even number $\geq 2$. 
\item $n\mid \sum d_i$.
\end{enumerate}
\end{proposition}
\begin{proof}
Since the inertia groups are cyclic subgroups of $G$, by description of the subgroups of $D_n$ in Remark \ref{subgroups of dihedral}, we see that the local monodromy for every point $t_i\in\{t_1,\dots,t_l\}$ is given by (the conjugacy class of) a reflection $a^{k}b$ and the rest of branch points correspond to elements of the form $a^{d_i}$ with $d_i|n$. Theorem \ref{Riemann existence} shows that there must exist reflections $a^{k}b$ among
the conjugacy classes arising from the branch points, for otherwise every generator belongs to the subgroup $\langle a\rangle$ and these can not generate $D_n$ so $l>0$. Now consider the epimorphism
\[\Phi:\pi_1(\P^1\setminus \br(f))\to G\]
Suppose $\gamma_{\infty},\gamma_2,\dots, \gamma_r$ are closed paths, where each $\gamma_i$ is a closed path around a $t_i$ and $\gamma_{\infty}$ is a closed path around the infinity and the infinity is not a branch point if and only if $\Phi(\gamma_{\infty})=1$. By Theorem \ref{Riemann existence}, this is equivalent to $\Phi(\gamma_{\infty})=\Phi(\gamma_{2})^{-1}\cdots \Phi(\gamma_{r})^{-1}=1$. But since $\langle a\rangle$ and $\langle b\rangle$ are disjoint subgroups, such a product can only
be trivial if there are an even number of elements of the form $a^{k}b$. If this is the case, we obtain that this is equal to $\Phi(\gamma_{2})\cdots \Phi(\gamma_{r})=a^{\sum d_i}$. Since $\ord(a)=n$, this is trivial if and only if $n\mid \sum d_i$.
\end{proof}

Let $B$ be a divisor with $\deg(B)=b$, an even number on $\mathbb{P}^1$ and $\mathcal{L}$ a line bundle on $\mathbb{P}^1$ such that $\mathcal{L}^2=\mathcal{O}_{\mathbb{P}^1}(B)=\mathcal{O}_{\mathbb{P}^1}(b)$.
Set $Z=\spec(\mathcal{O}_{\mathbb{P}^1}\oplus \mathcal{L})$. This is a smooth curve which  comes equipped with a covering map $q:Z\to\mathbb{P}^1$ branched along the divisor $B$ and
$q_*\mathcal{O}_{Z}=\mathcal{O}_{\mathbb{P}^1}\oplus \mathcal{L}$. By \cite{BHPV}, Lemma 17.1, $\Omega_Z=q^*(\Omega_{\mathbb{P}^1}\otimes \mathcal{L})=q^*\mathcal{O}_{\mathbb{P}^1}(\frac{b}{2}-2)$. \par Now let $D$ be a divisor on $Z$ together with a line bundle $\mathcal{F}$ such that $\mathcal{F}^2=\mathcal{O}_{Z}(D)$. Let $\mathcal{F}_i=\mathcal{F}^i(\langle \frac{iD}{n}\rangle)$ and $C=\spec(\oplus \mathcal{F}_i)$. There is a covering map $p:C\to Z$.
\begin{equation}
\begin{aligned}
\pi_*\Omega_C=q_*(p_*\Omega_C)=q_*(\Omega_Z\otimes p_*\mathcal{O}_C) =q_*(q^*\mathcal{O}_{\mathbb{P}^1}(\frac{b}{2}-2)\otimes (\oplus \mathcal{F}_i))= \\
q_*((\oplus \mathcal{F}_i)\otimes \mathcal{O}_{\mathbb{P}^1}(\frac{b}{2}-2)=q_*(\oplus \mathcal{F}_i)\otimes \mathcal{O}_{\mathbb{P}^1}(\frac{b}{2}-2)=\\
(\oplus U_i)\otimes
\mathcal{O}_{\mathbb{P}^1}(\frac{b}{2}-2)=\oplus (\Omega_Y\otimes
\mathcal{L}\otimes U_i).
\end{aligned}
\end{equation}
In the above we have used the projection formula (fourth equality). \par Note that the space $H^0(C,\mathcal{M}^{(1)}_{C})$ of meromorphic differential forms
on $C$ is a 1-dimensional vector space over the function field $\mathbb{C}(C)$ with basis $dx$. The equation \ref{dihedral equation} gives $dx=(\frac{2G^{\prime}x^n+nF^{\prime}F^{n-1}}{2nx^{n-1}(x^{n}-G)})dz$ and hence it follows that $H^0(C,\mathcal{M}^{(1)}_{C})=\{A(x)dz\mid A(x)\in\mathbb{C}(z)(x)\}$. Consequently,  it follows from \ref{FF eigenspace} that $H^0(C,\mathcal{M}^{(1)}_{C})_{\rho_h}=\langle x^hdz,x^{n-h}dz\rangle$ as a vector space over $\C(z)$. 

\end{document}